\documentclass[12pt,reqno]{amsart}
\usepackage{a4,amssymb,amsthm,amscd,amsmath,verbatim,url,enumerate,mathdots, hyperref,
cleveref,mathtools}
\usepackage[pdftex,usenames,dvipsnames]{xcolor}
\usepackage{fancyhdr}
\usepackage{tikz}

\title{A ruled residue theorem for function fields of hyperelliptic curves}
\author{Parul Gupta}
\author{Sumit Chandra Mishra}
\email{parul.gupta1@snu.edu.in, parulgupta11@gmail.com}
\email{sumitcmishra@gmail.com, sumitcmishra@iiti.ac.in}

\date{\today}

\newcommand{\nat}{\mathbb{N}} 
\newcommand{\zz}{\mathbb Z}

\newcommand{\mc}[1]{\mathcal{#1}}
\newcommand{\mf}[1]{\mathfrak{#1}}

\newcommand{\mg}[1]{{#1}^{\times}}
\newcommand{\sq}[1]{{#1}^{\times 2}}

\newcommand{\ovl}{\overline}

\newcommand{\car}{\mathsf{char}}

\newcommand{\mult}{\mathrm{mult}}

\renewcommand{\deg}{\mathsf{deg}}

\renewcommand{\max}{\mathsf{max}}
\renewcommand{\min}{\mathsf{min}}

\newcommand{\mfm}{\mf m}
\renewcommand{\setminus}{\smallsetminus}

\renewcommand{\leq}{\leqslant}
\renewcommand{\geq}{\geqslant}

\swapnumbers

\newtheorem*{thm*}{Theorem}

\newtheorem{thm}{Theorem}
\numberwithin{thm}{section}
\newtheorem{prop}[thm]{Proposition}
\newtheorem{cor}[thm]{Corollary}

\newtheorem{lem}[thm]{Lemma}

\theoremstyle{definition}
\newtheorem{de}[thm]{Definition}
\newtheorem{nota}[thm]{Notation}
\newtheorem{ex}[thm]{Example}

\newtheorem{rem}[thm]{Remark}

\newtheorem{obs}[thm]{Observation}

\newtheorem{assume}[thm]{Assumption}

\numberwithin{equation}{thm}
\renewenvironment{proof}{\par\noindent {\em Proof:}}{\hfill$\Box$\medskip}
\theoremstyle{plain}

\begin{document}

\begin{abstract}
We study residually transcendental extensions of a valuation $v$ on a field $E$ to function fields of hyperelliptic curves over $E$.
We show that $v$ has at most finitely many extensions to the function field of a hyperelliptic curve over $E$, for which the residue field extension is transcendental but not ruled, assuming that the residue characteristic of $v$ is either zero or greater than the degree of the hyperelliptic curve.

\medskip
\noindent
{\sc{Classification (MSC 2010):}} 12F20, 12J10, 12J20, 14H05, 16H05 


\end{abstract}

\maketitle

\section{Introduction}\label{intro}

Let $F/E$ be a function field in one variable. 
We say that $F/E$ is \emph{rational} if $F=E(x)$ for some element $x\in F$, and we say that $F/E$ is \emph{ruled} if $F=E'(x)$ for some finite extension $E'/E$ and some element $x\in F$. We say that $F/E$ is \emph{non-ruled} if it is not ruled.  

Let $E$ be a field with a valuation $v$ on it with value group ~$vE$. 
We say that a valuation $w$ on $F$ is an \emph{  extension of $v$ to $F$} if $w|_E=v$. 
An extension $w$ of $v$ to $F$ is called \emph{residually transcendental} if the residue field extension $Fw/Ev$ is transcendental. 
In 1983, J.~Ohm \cite{Ohm} proved that if $F/E$ is ruled, then for every residually transcendental extension $w$ of $v$ to $F$, $Fw/Ev$ is also ruled. 
This is called the \emph{Ruled Residue Theorem}. In the case where $vE=\zz^n$, this result was obtained in 1967 by M.~Nagata \cite[Theorem 1]{Nag}.

Let $F/E$ be a function field in one variable, not necessarily ruled. In the case where $vE=\zz$, it was observed in \cite[Corollary 3.9]{BGVG} that there are at most finitely many residually transcendental extensions $w$ of $v$ to $F$ such that $Fw/Ev$ is non-ruled. Moreover, assuming that $Ev$ is perfect and $E$ is relatively algebraically closed in $F$, it was proved in \cite[Theorem 5.3]{BGr} that the number of such residually transcendental extensions is bounded by $\mf{g}(F/E)+1$, where $\mf{g}(F/E)$ is the genus of $F/E$.
The case of a valuation $v$ with arbitrary value group was considered for the function fields of conics, and elliptic curves in \cite{BG21} and \cite{RRTEC} respectively. In case of the function field of conics (of elliptic curves, resp.), it was shown that if  $\car{Ev}\neq 2$ (if $\car{Ev}\neq 2, 3$, resp.), then there can be at most one residually transcendental extension $w$ of $v$ to $F$ such that $Fw/Ev$ is non-ruled. The case of the function fields of curves of prime degree was studied in \cite{AD}.

It is natural to consider next the case of function fields of hyperelliptic curves, which we undertake in this article. We prove that if $F/E$ is the function field of a hyperelliptic curve of degree $n$ ($n\geq 2$), and $\car{Ev}=0$ or $\car{Ev}>n$, then there are at most finitely many residually transcendental extensions $w$ of $v$ to $F$ such that $Fw/Ev$ is non-ruled (see \Cref{main-theorem}). We emphasize that no assumptions are made on the value group $vE$.

\begin{nota}
   Let $K$ be a field and let $v$ be a valuation on $K$.
We denote by $\mc{O}_v$ the valuation ring, by $\mfm_v$ its maximal ideal, by $Kv$ the residue field $\mc{O}_v/\mfm_v$, and by $vK$ the value group of $v$.
We denote the group operation in $vK$ additively. For $a\in {\mc O_v}$ we denote the residue class $a+\mfm_v$ in $Kv$
 by $av$, or just by 
 $\ovl{a}$ if the context makes this unambiguous. For a polynomial $f(t)=a_nt^n+\cdots + a_1t+ a_0 \in \mc O_v[t]$, we use $\ovl{f}(t)$ to denote its residue polynomial  $\ovl{a_n} t^n + \cdots +\overline{a_1}t + \overline{a_0} \in Kv[t]$.

 For a ring $R$, the set of all units of $R$ is denoted by $\mg{R}$.

\end{nota}

\subsection*{Outline of the proof} Let $f(X) \in E[X]$ be a square-free polynomial. 
Consider the hyperelliptic curve $\mc H : Y^2 = f(X)$. Let $F$ be the function field of $\mc H$ over $E$, that is, $F =E(X)[\sqrt{f(X)}]$. Let  $\Omega_v(F)$ be the set of all  residually transcendental extensions of $v$ to $F$ and $\Omega^\ast_v(F)$ be the subset of $\Omega_v(F)$ containing all those valuations $w$ for which $Fw/Ev$ is non-ruled.  

Let $f = \sum_{i=0}^na_it^i\in E[t]$ and $X\in F$ be a transcendental element over $E$ (e.g. $f$ and $X$ as above). Let $\Omega\subseteq \Omega_v(F)$.
We call a valuation $w \in \Omega$ \emph{overt with respect to $(f,X)$} if $w(f(X)) = \min\{w(a_iX^i) \mid 0\leq i \leq n\}$ and we say that $w$ is \emph{covert with respect to $(f,X)$}, otherwise. The set of all overt (resp. covert) valuations with respect to $(f,X)$ in $\Omega$ is denoted by $\mathsf{Overt}_{f,X}(\Omega)$ (resp. $\mathsf{Covert}_{f,X}(\Omega)$). 
Note that $\Omega$ can be considered as a disjoint union of overt and covert extensions with respect to $(f,X)$. Our proof is based on a careful analysis of both these sets for $\Omega =\Omega^\ast_v(F)$.

In \Cref{overt}, we show that  $|\mathsf{Overt}_{f,X}(\Omega_v^\ast(F))| \leq n^2$.
We give an inductive proof for the finiteness of non-ruled covert extensions.

For a given valuation $w \in \Omega_v(F)$, we associate a subset $S_{f, X}(w)$ of $\{0,1,\ldots,n\}$ as follows:
$$S_{f, X}(w) =\{ i \mid w(a_iX^i) = \min\{ w(a_iX^i)\mid 0\leq i \leq n\}\}.$$ 
We observe that if $Fw/Ev$ is non-ruled, then $|S_{f,X}(w)| \geq 2$ (see \Cref{minimum-attained-twice}). 
 For each set $S \subseteq \{0, 1, \ldots, n\}$ and $|S| \geq 2$, we consider $$\Omega_{S,f,X} = \{w \in \Omega_v (F) \mid S =S_{f,X}(w)\}.$$ 
Then $\mathsf{Covert}_{f,X}(\Omega_{S,f,X})$ is the set  of all covert extensions with respect to $(f,X)$ such that $S_{f,X} (w) =S$. 
As there are only finitely many subsets of $\{0, 1, \ldots, n\}$, it is enough to show that each of the sets $\mathsf{Covert}_{f,X}(\Omega_{S,f,X}) \cap \Omega_v^\ast(F)$ is finite.

 Let  $0 \leq i_0 < i_1 < \cdots < i_k\leq n$ and  $S = \{ i_0, i_1, \ldots, i_k\}$. Assume that $k\geq 1$, i.e. $|S|\geq 2$.
We show that there is an integer $e_S \in \nat$ and an element $c_S \in \mg E$, 
such that $w(c_SX^{e_S}) =0$ for all $w\in \Omega_{S,f,X}$ (see the paragraph after \Cref{cS}).
 We emphasize that $e_S$ and $c_S$  depend only on $S$ and $f$, and do not depend on $w$. Note that such an element $c_S$ can be replaced by any element in $c_S\mg{\mc O}_v$.
For $S,f$ and $c_S$, we associate a polynomial over the residue field $Ev$ as follows:
$$p_{S,f,c_S}(t) = \sum_{j=0}^{k} \ovl{a_{i_j} a_{i_0}^{-1}c_S^{-n_j}} t^{n_j},~\mbox{ where } n_j = \frac{i_j-i_0}{e_S}.$$ 
We show that
$p_{S,f,c_S}(\ovl{c_SX^{e_S}}) =0$ in $Fw$ for all $w \in \mathsf{Covert}_{f,X}(\Omega_{S,f,X})$, see \Cref{residue-root}. 
Let $\mathsf{Mult}({S,f,X} )$ be the set of multiplicities of all the roots of $p_{S,f,c_S}$. We note that $\mathsf{Mult}({S,f,X} )$ does not depend on the choice of $c_S$. All these notions are defined and explained in \Cref{4}.

Our proof is based on the induction on $\max~ \mathsf{Mult}({S,f,X} )$. 
In \Cref{5}, we develop a multiplicities reduction technique.
Starting with a given pair $(f,X)$ and a subset $S$, 
for each monic irreducible factor $q$ of $p_{S,f,c_S}$, we construct a base extension $E'$  with a valuation extension $v'$ of $v$ to $E'$ 
such that $\frac{1}{e_S}v(c_S) \in v'E'$. 
The construction of $E'$ depends only on the polynomial $f,S$ and $q$. 
On the function field $F' =FE'$, we find a set of valuations $\Omega$ such that for each  $w \in \mathsf{Covert}_{f,X}(\Omega_{S,f,X})$ with $q(\ovl{c_SX^{e_S}}) =0$ in $Fw$, there is a unique extension $w'$ of $w$ to $F'$ in $\Omega$ 
satisfying a couple of properties, but most importantly $F'w' =Fw$.
Further, we find a polynomial $g\in E'[t]$ and an element $X_1 \in F'$ such that $F' =E'(X_1)[\sqrt{g(X_1)}]$ and 
for each $w\in \Omega$,  $S_1:=S_{g,X_1}(w) $ 
where $\max \;\mathsf{Mult}(S_1,g,X_1) < \mathsf{Mult}(S,f,X)$. 
We prove the base case for the induction is proven in \Cref{base-case}, that is, if $\max \;\mathsf{Mult}(S,f,X) =1$ then $\Omega_{S,f,X}\cap \Omega^{\ast}_v(F)$ is finite.

\section{Base field extensions}\label{2}

Let $K$ be a field and $v$ be a valuation on $K$. Let $L/K$ be a field extension.
An \emph{extension of $v$ to $L$} is a valuation $w$ on $L$ such that $vK$ is an ordered subgroup of $wL$ and $w|_K=v$. By Chevalley's Theorem \cite[Section 3.1]{EP}, $v$ always extends to a valuation on $L$. 
Two valuations $w$ and $w'$ on $L$ are \emph{distinct} if $\mc O_w \neq \mc O_{w'}$.

In this section, we begin by establishing some results concerning the extension of valuations to finite simple extensions. These results will then be applied to valuations on regular function fields $F/E$ of transcendence degree one. Specifically, we consider residually transcendental extensions $w$ of the base valuation $v$ from $E$ to $F$. We construct suitable extensions 
$E'/E$ of the base field such that there exists a valuation extension $w'$ of $
w$ to $FE'$ satisfying  $(FE')w' = Fw$. In this context, if the residue field extension 
$Fw/Ev$ is non-ruled, then the same holds for 
$(FE')w'/Ev$.

We begin by recalling the Fundamental Inequality.

\begin{thm}[Fundamental Inequality]
\label{funda_ineq}
Let $K$ be a field and $v$ be a valuation on $K$.
Assume that $L/K$ is a finite field extension. 
Let $r\in\nat$ and let $w_1,\dots, w_r$ be distinct extensions of $v$ to $L$.
Then
$$ \displaystyle \sum^{r}_{i=1} [w_iL:vK] \cdot [Lw_i:Kv] \leq [L:K].$$
\end{thm}

\begin{proof}
See \cite[Theorem 3.3.4]{EP}.
\end{proof}

A polynomial $f(t) = a_nt^n+\cdots + a_1t+ a_0 \in \mc O_v[t]$ is called \emph{primitive} if $\ovl{f}(t)\neq 0 
$, that is,  there exists at least one $i \in \{0, \dots ,n\}$ such that $a_i$ is a unit in $\mc O_v$. 
Note that for a polynomial $f(t) =a_nt^n+\cdots + a_1t+ a_0 \in K[t]$ and  for $0\leq i_0 \leq n$ such that $v(a_{i_0}) = \min \{v(a_i)\mid 0 \leq i\leq n\}$, the polynomial $a_{i_0}^{-1} f(t) \in \mc O_v[t]$ is primitive.


\begin{lem}\label{simple_ext}
 Let $f\in\mc{O}_v[t]$ such that $\ovl{f}\in Kv[t]\setminus Kv$. 
 Let $q \neq t$ be an irreducible factor of $\ovl f$ in $Kv[t]$.
Then there is an irreducible factor $g \in\mc{O}_v[t]$ of $f$ such that $q$ is an irreducible factor of $\ovl{g}$. Let $L=K[t]/(g)$ and $\alpha \in L$ be a root of $g$. Then 
\begin{enumerate}[$(1)$]
    \item there is an extension $v'$ of $v$ to $L$ such that $v'(\alpha) =0$ and $q(\ovl{\alpha}) =0$ in $Lv'$. Moreover, if $\ovl{\alpha}$ is a simple root of $\ovl f$, then $[Lv':Kv] =\deg(q)$, $Lv'=Kv[\ovl{\alpha}]$ and $v'L = vK$.
    \item  if $\ovl{g}$ is separable and $\deg(\ovl{g}) = \deg(g)$, then the extension $v'$ in $(1)$ is unique.
    \item if $\ovl{g}$ is linear, then the extension $v'$ in $(1)$ is unique.
    
\end{enumerate}
\end{lem}

\begin{proof}
Let $g_1(t), \ldots, g_r(t) \in K[t]$ be all the monic irreducible factors of $f(t)$, not necessarily distinct. 
 For $1\leq i\leq r$, let $b_i \in \mg E$ be such that $b_ig_i(t) \in \mc O_v[t]$ is primitive. 
Let $a$ be the leading coefficient of $f$. We have $$f(t) = a(b_1\ldots b_r)^{-1}\prod_{i=1}^r b_ig_i(t) $$
and 
$$\ovl{f}(t) = \ovl{ a(b_1\ldots b_r)^{-1}}\prod_{i=1}^r \ovl{b_ig_i}(t).$$
Since $q$ divides $\ovl{f}$, we have that $q$ divides $\ovl{b_ig_i} \in \mc O_v[t]$ for some $1\leq i \leq r$.

\medskip

$(1)$ By \cite[Lemma 2.3]{RRTEC}, there is an extension $v'$ of $v$ to $L$ such that $ v'(\alpha) = 0 $ and $q(\ovl{\alpha}) =0$ in $Lv'$. 
Thus $Ev[\ovl{\alpha}] \subseteq Lv'$.
Furthermore, given an extension $v'$ of $v$ to $L$ with $v'(\alpha) = 0$, if $\ovl{\alpha}$ in $Lv'$ is a simple root $\ovl{f}$, then $v'L=vK$ and $Lv' =Kv[\ovl{\alpha}]$. Thus $[Lv': Kv] = \deg(q)$.

$(2)$ Suppose further that $\ovl{g}$ is separable and $\deg(\ovl{g}) = \deg(g)$. Write $\ovl{g} = q_1\cdots q_r$, where $q_1, \ldots, q_r \in Kv[t]$ are distinct irreducible with $q_1=q$. 
Let $1\leq i\leq r$. By $(1)$,  there is an extension $v_i$ of $v$ to $L$ such that $v_i(\alpha) =0$, $q_i(\ovl{\alpha}) =0$ in $Lv_i$ and $[Lv_i: Kv] = \deg(q_i)$. 
For an extension $\tilde v$ of $v$ to $L$ with $\tilde v(\alpha) =0$ and $q_i(\ovl{\alpha}) =0$ in $L\tilde v$, we have that $L\tilde v$ contains a root of $q_i$, and hence $[L\tilde v:Kv] \geq \deg(q_i)$.
Let $n_i$ be the number of valuation extensions $\tilde v$ of $v$ to $L$ with $
\tilde v(\alpha) =0$ and $q_i(\ovl{\alpha}) =0$ in $L\tilde v$. 
We have $$ \sum_{i=1}^r n_i\deg(q_i)  \leq  \sum_{v'|_K=v} [Lv': Kv] \leq [L: K] = \deg(g) = \deg{(\ovl{g})} = \sum_{i=1}^r \deg(q_i),$$ where the second inequality is implied by the Fundamental inequality.
From the above equation we get $n_i =1$ for all $1\leq i\leq r$. Thus we get that there is a unique extension $v'$ of $v$ to $L$ such that $v'(\alpha) =0$ and $q(\ovl{\alpha}) =0$ in $Lv'$.

$(3)$ Suppose further that $\ovl{g}(t)$ is linear. Let $N$ be a normal closure of $L$ and let $v''$ be an extension of $v'$ from $L$ to $N$. Write
$$g(t) = c \prod_{i=1}^d (t-\alpha_i)$$
where  $\alpha_1, \ldots, \alpha_d \in N$ with and $c\in K$ is the leading coefficient of $g(t)$. We further assume that
$v''(\alpha_1) \geq  v''(\alpha_2)\geq \cdots \geq v''(\alpha_d)$. 
If $v(\alpha_d)\geq 0$ then let $d'=d$; otherwise,  let  $d'$ be the smallest natural number such that $v''(\alpha_i) <0$ for all $i > d'$.
Then 
$$\ovl{g}(t) = \ovl{ c(\alpha_{d'+1}\ldots \alpha_d)}\prod_{i=1}^{d'} (t -\ovl{\alpha_i})(-1)^{d-d'} .$$

Since  $\ovl{g}(t) \in Nv''[t]$ is linear and $v''(\alpha) =0$, we obtain that $d' =1$, $\alpha_1 =\alpha$ and $v(\alpha_i) <0$ for all $i >1$. 

Let $\tilde v$ be an extension of $v$ to $N$. Then $\tilde v = v''\circ \sigma$, where $\sigma$ is a $K$-automorphism of $N$ \cite[Conjugation Theorem 3.2.15]{EP}. Since $L =K[\alpha]$, we have that if $\sigma(\alpha) =\alpha$ then $\sigma|_L$ is identity and hence $\tilde v|_L = v'$. Suppose now that $\sigma(\alpha) \neq \alpha$.
Then $\tilde v(\alpha) = v''\circ \sigma (\alpha) = v''(\sigma(\alpha)) <0$. Thus $\tilde v|_L\neq v'$. Thus in this case $v$ extends to a unique extension $v'$ to $L$ such that $v'(\alpha) =0$ and $\ovl{g}(\ovl{\alpha}) =0$ in $Lv'$.
\end{proof}

For the rest of the section, we fix a field $E$ with a valuation $v$ on $E$. Let $F/E$ be a regular function field in one variable, i.e. $E$ is relatively algebraically closed in $F$.
Let $w$ be a residually transcendental extension of $v$ to $F$. 

We note a well-known statement below.

\begin{rem}\label{regularfunctionfields}
Let $Q$ be a monic irreducible polynomial over $E$. Since all roots of $Q$ are algebraic over $E$, all the coefficients of  an irreducible factor of $Q$ in $F[t]$ are also algebraic over $E$ and hence belong to $E$. Thus $Q$ remains irreducible over $F$. Therefore, $F[t]/(Q(t))$ is a field extension of $F$. 
\end{rem}

\begin{nota}
$(1)$ For an irreducible polynomial $Q \in E[t]$, we set $F_Q = F[t]/(Q(t))$ and call it the root field of $Q$ over $F$. For $\alpha \in F_Q$ a root of $Q(t)$, we set $E_Q = E[\alpha]$. Then $E_Q/E$ is a finite extension and $F_Q =FE_Q$.
\\
$(2)$ Let $e \in \nat$ and $a\in \mg E$ be such that $t^e-a$ is irreducible in $E[t]$, then the root field of $t^e-a$ is written as $F[\sqrt[e]{a}]$. 
\end{nota}

\begin{lem}\label{baseextension}
    Let $\beta \neq 0 \in Fw$ be algebraic over $Ev$ and  $q(t)\in Ev[t]$ be the minimal polynomial of $\beta$ over $Ev.$ Let $Q(t)\in \mc{O}_v[t]$ be monic such that $\ovl{Q}(t)=q(t).$ Assume that $\car Ev $ does not divide $\deg(q)$.
    Then $Q(t)$ is irreducible over $F$.
Let $F_Q = F[t]/(Q(t))$, $\alpha =t+(Q(t)) \in F_Q$ and $E_Q=E[\alpha]$. 
Then $v$ extends uniquely to a valuation $v_Q$ on $E_Q$ and $v_Q(\alpha) =0$. Furthermore, there is a unique extension $w_Q$ of $w$ to $F_Q$ with $\ovl{\alpha} = \beta$ in $F_Q w_Q$. Furthermore, for this extension $w_Q$, we have $F_Q w_Q=Fw$ and $w_QF_Q =wF$.

\end{lem}

\begin{proof}
Let $q(t) =t^m +\beta_{m-1}t^{m-1} + \ldots + \beta_0$ where $\beta_i\in Ev$ for $0\leq i \leq m-1$. Then $Q(t)=t^m+ b_{m-1}t^{m-1} + \ldots +b_0$ where $b_i\in \mc O_v$ and $\ovl{b_i}=\beta_i$ for $0\leq i\leq m-1.$ 
Since $\ovl{Q}(t) =q(t)$ and $q(t)$ is irreducible in $Ev[t]$, we get that $Q(t)$ is irreducible in $E[t]$. By \Cref{regularfunctionfields}, we get that $Q(t)$ is irreducible in $F[t]$.
Thus $[F_Q: F] = [E_Q:E]= m =[Ev[\beta]: Ev]$. 
Using the fundamental inequality, we get that $v$ extends uniquely to $E_Q$. 
Let $v_Q$ be the unique extension of $v$ to $E_Q$. Then $v_Q(\alpha^m + b_{m-1}\alpha^{m-1}+ \ldots + b_1\alpha + b_0) = v(Q(\alpha)) > 0$. 
Since $v(b_0) =0$, we must have $v_Q(\alpha^m + b_{m-1}\alpha^{m-1}+ \ldots + b_1\alpha)=0$. Note that if $v_Q(\alpha)>0$ then $v_Q(\alpha^m + b_{m-1}\alpha^{m-1}+ \ldots + b_1\alpha)\geq v_Q(\alpha)>0$, which is a contradiction. Similarly, if $v_Q(\alpha)<0$ then $v_Q(\alpha^m + b_{m-1}\alpha^{m-1}+ \ldots + b_1\alpha)= m v_Q(\alpha)<0$, which is a contradiction. Thus we conclude that $v_Q(\alpha)=0$.

We have $F_Q =F[\alpha]$.  
Since $\car Ev $ does not divide $m$, $\ovl{Q}(t) = q(t)$ is separable.
Since $q(t)$ is the minimal polynomial of $\beta$ over $Ev$, $t-\beta $ is an irreducible factor of $q(t)$ in $Fw[t]$.
By \Cref{simple_ext}$(1)$, we get that there exists an extension $w_Q$ of $w$ to $F_Q$ such that $w_Q(\alpha) =0$, $\ovl{\alpha}=\beta$ in $F_Q$, $w_QF_Q=wF$, and $F_Qw_Q = Fw[\ovl{\alpha}] = Fw[\beta]=Fw$ as $\beta \in Fw$.

Furthermore, since $\deg(Q(t)) = \deg(q(t))$,  by \Cref{simple_ext}$(2)$, such a valuation $w_Q$ is unique.
\end{proof}


 \begin{lem}\label{baseextension3} 
Let $X \in F \setminus E$ and let $e$ be the order of $w(X)$ in $wF/vE$. Let $c \in \mg E$ be such that $w(cX^e) =0$.
Assume that $\car{Ev}$ does not divide $e$ and $\ovl{cX^e} \in Ev$. Let $\alpha \in \mg{\mc O}_v$ be such that $\ovl{\alpha} =  \ovl{cX^e}$. 
Then  $p(t)= t^e - \alpha^{-1}cX^e  \in \mc O_w[t]$ is primitive irreducible polynomial with the residue polynomial $t^e-1$ in $Fw[t]$.
Let $F' = F [\sqrt[e]{\alpha^{-1}c}]$ and let  $\lambda \in F'$ be such that $\lambda^e =\alpha^{-1}c$.
Then there is a unique extension $w'$ of $w$ to $F'$ such that $w' (\lambda X) =0$, $ \ovl{\lambda X} =1$ in $F'w'$, and for the valuation $w'$, $F'w'=Fw$.
\end{lem}

\begin{proof}  
Let $E' =E[\lambda]$. The minimal polynomial of $\lambda$ over $E$ divides $t^e- \alpha^{-1}c$, 
thus $[E':E] \leq e$. Let $v'$ be an extension $v$ to $E'$. 
Then $v'(\lambda) = \frac{v(c)}{e}$. 
Since $e$ is the order of $w(X) = -\frac{v(c)}{e}$ in $wF/vE$, 
we get that $e$ is also  the order of $\frac{v(c)}{e}$ in $v'E'/vE$. 
Thus $e \leq [v'E' :vE] \leq [E':E] \leq e $, whereby the minimal polynomial of 
$\lambda $ over $E$ is equal to $t^e -\alpha^{-1}c$. 
By \Cref{regularfunctionfields}, we get that $t^e -\alpha^{-1}c$ is irreducible over $F$. 
Since $X\in F$, we have that $F' = F[\lambda] = F[\lambda X] $, 
whereby $p(t) =t^e -  \alpha^{-1} cX^e$ is the minimal polynomial of $\lambda X$ 
and hence irreducible.

Since $w(\alpha) =0 $ and $w(cX^e) =0$, we get that $w(\alpha^{-1}cX^e) =0$. 
Since $\ovl{\alpha} = \ovl{cX^e}$ in $Fw$, we get that the residue polynomial is 
$$\ovl{p}(t) =t^e - \ovl{\alpha^{-1}cX^e} = t^e -1 \in Fw[t].$$ 

Since $t-1$ is an irreducible factor of $t^e-1$, by \Cref{simple_ext}$(1)$, there exists an extension $w'$ of $w$ to $F'$ such that $w'(\lambda X) =0$ and $\ovl{\lambda X} =1$. Furthermore, since $\car{Ev}$ does not divide $e$, $t^e -1$ is separable over $Ev$, hence we get by \Cref{simple_ext}$(2)$ that such an extension $w'$ is unique. Since $\ovl{\lambda X} = 1$ is a simple root of $t^e -1 $, thus $F'w' =Fw$ by \Cref{simple_ext}$(1)$.
\end{proof}

\section{Ruled residue extensions}\label{3}

In this section, we review some results related to extensions of valuations from a base field $E$ to function fields of transcendence degree one over $E$. This includes the characterization of the Gauss extension, and the Ruled Residue Theorem, etc. Furthermore, we revisit some essential intermediate results from \cite{BG21} and \cite{RRTEC}, which were instrumental in establishing the Ruled Residue Theorem for function fields of conics and function fields of elliptic curves. These results will be used frequently in proving our main result.

For the rest of the section, we fix a field $E$ and a valuation $v$ on $E$.
Let $F/E$ be a function field in one variable over $E$. 
An extension $w$ of $v$ to $F$ is called \emph{residually transcendental} if the residue field extension $Fw/Ev$ is transcendental.
The \emph{set of all residually transcendental extensions} of $v$ to $F$ is \textbf{$\Omega_v(F)$}.

A well-known example of a residually transcendental extension to the rational function field $E(X)$ in one variable $X$ is the Gauss extension of $v$ with respect to $X$, which is characterized in the following proposition.   

\begin{prop}\label{gaussextdef}
There exists a unique valuation $w$ on $E(X)$ with $w|_E=v$, $w(X) = 0$ and such that the residue $\ovl{X}$  of $X$ in $E(X)w$ is transcendental over $Ev$. 
For this valuation $w$, we have that $E(X)w =Ev(\ovl{X})$ 
and $wE(X) = vE$.  
\end{prop}

\begin{proof} 
See \cite[Corollary 2.2.2]{EP}.
\end{proof}

We call a valuation on $E(X)$ a \emph{Gauss extension of $v$ to $E(X)$} if it is the Gauss extension with respect to $Z$ for some  $Z\in E(X)$ with $E(Z)=E(X)$, that is, $w(Z) =0$ and $\ovl{Z}$ is transcendental over $Ev$.

Let $F/E$ be a function field in one variable over $E$. Note that for a residually transcendental extension $w$ of $v$ to $F$, there always exists $Z\in \mc O_w$ such that $\ovl{Z}$ is transcendental over $Ev$. For such a $Z$, $w|_{E(Z)}$ is the Gauss extension with respect to $Z$.

\begin{thm}[Ohm]\label{RRT}
Let $w$ be a residually transcendental extension of $v$ to $E(X)$.
Then $E(X)w/Ev$ is ruled.
\end{thm}
\begin{proof}
See \cite[Theorem 3.3]{Ohm}).
\end{proof}

In the rest of the section, we consider function fields $F$ that are quadratic extensions of the rational function field $E(X)$. We will assume that $v(2) =0$, that is $\car{Ev} \neq 2$. 
Thus $\car E\neq 2$ and $F =E(X)[\sqrt{f(X)}]$ for a square-free polynomial $f(X) \in E[X]$.

\begin{obs} \label{Conseq_RRT} Assume that $v(2) =0$. Let $F/E$ be a function field in one variable over $E$. Let $w \in \Omega_v(F)$. Thus $Fw/Ev$ is transcendental.  Let $X \in F\setminus E$. If $Fw =E(X)w$, then $Fw/Ev$ is ruled by the Ruled Residue Theorem.
Assume that $[F:E(X)] =2$. 
\begin{enumerate}
 \item  If $Fw/Ev$ is non-ruled, then using the fundamental inequality (see \cite[Corollary 2.2]{RRTEC}) we get that 
   $[Fw: E(X)w] =2$ and $wF =wE(X)$.
   
   \item Let $\ell$ be the relative algebraic closure of $Ev$ in $Fw$. If $\ell \not \subseteq E(X)w$, then $Fw = \ell E(X)w$, whereby $Fw/Ev$ is ruled.
   Thus, if $Fw/Ev$ is non-ruled then $\ell \subseteq E(X)w$ (See \cite[Lemma 3.2]{BG21}).
\end{enumerate}

\end{obs}

 Suppose $F=E(X)[\sqrt a]$ for some $a\in E\setminus \sq E$ (or $F =E(X)( \sqrt Z)$ for some $Z \in E(X)$ such that $E(Z) =E(X)$ resp.). Then $F=E[\sqrt{a}](X)$ (or $F=E(Z)[\sqrt{Z}]=F(\sqrt{Z})$ resp.) and hence by the Ruled Residue Theorem, we get that for every residually transcendental extension $w$ of $v$ to $F$, $Fw/Ev$ is ruled. Now we consider quadratic extensions of the form $F =E(X)[\sqrt{au}]$ or $F =E(X)[\sqrt{Zu}]$ where $u\in E(X)$, and find certain extensions $w$ on $F$ for which $Fw/Ev$ is ruled.

 \begin{prop}
  \label{u_algebraic_implies_ruled}
 Assume that $v(2)=0$.
 Let  $u\in E(X)$ and $F=E(X)[\sqrt{uX}]$. 
Let $w \in \Omega_v(F)$ with  $w(u)=0$
and such that $\ovl{u}$ is algebraic over $Ev$. 
Then $Fw/Ev$ is ruled.   
\end{prop}
\begin{proof}
See \cite[Proposition 4.5]{RRTEC}.
\end{proof}

\begin{prop}
 \label{u_algebraic_implies_ruled 0}
 Assume that $v(2)=0$.
 Let $a \in  \mg E$, $u\in E(X)$ and $F=E(X)[\sqrt{au}]$. 
Let $w \in \Omega_v(F)$ with  $w(u)=0$
and such that $\ovl{u}$ is algebraic over $Ev$. 
Then $Fw/Ev$ is ruled. 
 \end{prop}

\begin{proof}
Let $\ell$ denote the relative algebraic closure of $Ev$ in $Fw$.
In view of \Cref{Conseq_RRT}, we may assume that $[Fw:E(X)w]=2$ and $wF =wE(X)$ and  $\ell\subseteq E(X)w$. 
We get that $w(a)=w(au)\in 2wF =2wE(X)$ and for every $\phi\in E(X)^{\times}$ with $w(a\phi^2)=0$, it follows by \cite[Lemma 4.1]{RRTEC} that $\ovl {u a\phi^2}$ is transcendental over $Ev$ and 
since $\ovl{u}\in \ell$, $\ovl {a\phi^2} $ is transcendental over $Ev$. 
Moreover, for every such $\phi$, if $\ovl{a\phi^2} \in  E(X)w^{\times 2}$, then by \cite[Corollary 2.5]{RRTEC} $Fw = E(X)w\left[\sqrt{\ovl{u a\phi^2}}\right] = E(X)w[\sqrt{\ovl{u}}]$, and since $\ell \subseteq E(X)w$ we obtain $Fw =E(X)w$. As $Fw \neq E(X)w$, we get that $\ovl{a\phi^2} \notin  E(X)w^{\times 2}$.

Let $w'$ be an extension of $w|_{E(X)}$ to $E(X)[\sqrt{a}]$.
Then $ E(X)[\sqrt{a}]w'=E(X)w\left[\sqrt{\ovl{a\phi^2}}\right]$ and hence $[E(X)[\sqrt{a}]w':E(X)w]=2$.
By \cite[Lemma 4.1]{RRTEC}, the relative algebraic closure of $Ev$ in $E(X)[\sqrt{a}]w'$ is contained in $E(X)w$ and hence is equal to $\ell$. 
By \cite[Lemma 2.8]{BG21},  there exists $\phi \in E(X)\setminus \{0\}$ with $w(a\phi^2) =0$ such that 
$E(X)[\sqrt{a}]w'=\ell(\ovl{\sqrt{a}\phi})$. Since $\ell \subseteq E(X)w$ and $[E(X)[\sqrt{a}]w':E(X)w]=2$, we have $E(X)w=\ell(\ovl{a \phi^2})$. Set $\vartheta =\ovl{a\phi^2}$. Using \cite[Lemma 2.2]{BG21}, we get that
$$Fw = E(X)w[\sqrt{\ovl{u}\vartheta}]=\ell(\vartheta)[\sqrt{\ovl{u}\vartheta}]\,.$$  
Since $\ovl{u}\in\ell$, we have $Fw=\ell(\vartheta)[\sqrt{\ovl{u}\vartheta}]=\ell (\sqrt{\ovl{u}\vartheta})$, which is ruled over $Ev$.
\end{proof}

Let $F = E(X)[\sqrt {f(X)}]$ where $f(X) \in E[X]$ is a square-free polynomial. 
 The following statement is a further observation on \cite[Corollary 4.4]{RRTEC}.
There $E(X)$ was viewed as a finite extension of $E(f)$, that is $E(X) = E(f)[\theta] $ for some $\theta\in E(X)$.
For a residually transcendental extension $w$ of $v$ to $F$, there the statement  required that the minimal polynomial of $\theta$ over $E(f)$ is a primitive polynomial in $\mc O_w[t]$, however the statement holds with the same proof even if we replace the minimal polynomial by a  scalar multiple of the minimal polynomial that is primitive. 

\begin{cor}
\label{ruled_applicable}
Assume that $v(2) =0$.
Let $w \in \Omega_v(F)$.
Let $f\in E[X]$. 
Let $\theta \in E(X)$ be such that $E(X) = E(f)[\theta]$ and let $p_\theta \in E(f)[t]$ be an irreducible primitive polynomial such that $p_\theta(\theta) =0$.
Assume that
$\ovl{\theta}$ is a simple root of $\ovl{p_\theta}$.
Then $wE(X) =wE(f)$ and $E(X)w =E(f)w[\ovl{\theta}]$.
Furthermore, if $\ovl{\theta}$ is  algebraic over $Ev$, then $E(X)[\sqrt{f}]w'/Ev$ is ruled for any extension $w'$ of $w$ to $E(X)[\sqrt{f}]$.
\end{cor}
\begin{proof}
See \cite[Corollary 4.4]{RRTEC}.
\end{proof}

\begin{lem}\label{distinctvalues}
Assume that $v(2)=0.$
Let $F/E$ is the function field of a hyperelliptic curve. 
Let $f \in E[t]$ and $X\in F\setminus E$ be such that  $F = E(X) [\sqrt{f(X)}]$. Let $f_1(X), f_2(X) \in E[X]$ be such that $f(X) = f_1(X) +f_2(X)$.
 Let $w\in \Omega_v(F)$ such that $w(f_1(X)) < w(f_2(X))$. Then either $Fw/Ev$ is ruled or $Fw = F'w'$ for some extension $w'$ of $w|_{E(X)}$ to $F' = E(X)[\sqrt{f_1(X)}]$.     
\end{lem}

\begin{proof}
The statement follows from \cite[Lemma 2.6]{RRTEC}.   
\end{proof}

\begin{nota}Let $n\in\nat$, $a_0, \ldots, a_{n-1} \in E$ and $a_{n} \in \mg E$. 
Let $f(t) = \sum_{i=0}^n a_it^i \in E[t]$ be a polynomial and let $X\in F\setminus E$.
For $w \in \Omega_v(F)$, we define
$$ S_{f,X}(w) = \{i\in \{0, 1, \dots, n\} \mid  w(a_iX^i)= \min \{w(a_jX^j)| 0\leq j \leq n\} \}.$$
Note that $S_{f,X}(w)\subseteq \{0, 1, \ldots, n\} $.
\end{nota}

\begin{cor}
\label{minimum-attained-twice}
Assume that $v(2)=0$. Let $F/E$ is the function field of a hyperelliptic curve. 
Let $f \in E[t]$ and $X\in F\setminus E$ be such that  $F = E(X) [\sqrt{f(X)}]$. 
Let $w \in \Omega_v(F)$.
If $|S_{f,X}(w)|=1$ then $Fw/Ev$ is ruled.
\end{cor}

\begin{proof}
Write $f(t ) = a_0+a_1t+ \ldots+ a_nt^n$ with $a_0, a_1, \ldots, a_{n-1} \in E$ and $a_n\in \mg E$.
Let $w\in \Omega_v(F)$. Assume that $S_{g,X}(w)=\{i_0\}$ for some $i_0\in \{0, 1, \dots, n\}$. Then $$w(a_{i_0}X^{i_0}) < w\left( \sum_{i=0,i\neq i_0}^{n} a_iX^i \right).$$ 
    By \Cref{distinctvalues}, either $Fw/Ev$ is ruled, or $Fw=E(X)[\sqrt{a_{i_0}x^{i_0}}] w'$ for some extension $w'$ of $w|_{E(X)}$ to $E(X)[\sqrt{a_{i_0}x^{i_0}}]$. Since $E(X)[\sqrt{a_{i_0}x^{i_0}}] w'/Ev$  is ruled, so is $Fw/Ev$.  
\end{proof}

\begin{prop}\label{overt}
Assume that $v(2)=0$. Let $F/E$ is the function field of a hyperelliptic curve.
Let $f(t) = \sum_{i=0}^n a_it^i\in E[t]$ of degree $n$, and $X\in F\setminus E$ be such that  $F = E(X) [\sqrt{f(X)}]$.
There are at most $n^2$ residually transcendental extensions $w$ of $v$ to $F$ such that 
$w(f(X)) =\min \{w(a_jX^j)| 0\leq j \leq n\}, $ and $Fw/Ev$ is non-ruled. 
\end{prop}

\begin{proof}
Let $w$ be a residually transcendental extension of $v$ to $F$.
Let $i_0 = \min \{S_{f,X}(w)\}$ and set 
$$Z = \frac {f(X) -a_{i_0}X^{i_0}}{a_{i_0}X^{i_0}}.$$
Then $f(X) = a_{i_0}X^{i_0} (1+ Z) $.  Note that $w(Z)\geq 0$.
Since $w(f(X))=w(a_{i_0}X^{i_0})$, we get that $w(1+Z) =0$.
 Thus  $ \ovl{1+Z} \neq 0$.
Suppose that $\ovl{Z}$ is algebraic over $Ev$.
Then $\ovl{1+Z}= 1+\ovl{Z}$ is algebraic over $Ev$.
If $i_0$ is even, then $F = E(X) [\sqrt{a_{i_0}(1+Z)}]$ and by \Cref{u_algebraic_implies_ruled 0}, $Fw/Ev$ is ruled.
If $i_0$ is odd, then $F = E(X) [\sqrt{a_{i_0}X(1+Z)}] $ and by \Cref{u_algebraic_implies_ruled}, $Fw/Ev$ is ruled.

Suppose now that $\ovl{Z}$ is transcendental over $Ev$. 
Then by \Cref{gaussextdef}, $w|_{E(Z)}$ is the Gauss extension with respect to $Z$. 
Since $[E(X): E(Z)] = n$, using the fundamental inequality, we get that there are at most $n$ extensions of $w|_{E(Z)}$ to $E(X)$. 
By \cite[Corollary 2.4]{RRTEC}, we get that there is a unique extension of $w|_{E(X)}$ to $F$ for which the residual field extension is non-ruled. 
Also, if $Fw/Ev$ is non-ruled then $|S_{f,X}(w)|\geq 2$ by \Cref{minimum-attained-twice}. Thus $i_0\neq n$ and we have $n$ many possibilities for $i_0$ and hence at most $n$ many possibilities for $Z$. 
Hence there are at most $n^2$ residually transcendental extensions $w$ of $v$ to $F$ such that $w(f(X))= \min\{w(a_i X^i) \mid 0\leq i \leq n\}$ and $Fw/Ev$ is non-ruled.
\end{proof}

Let $F/E$ be a regular function field of transcendence degree one. 
Let $n\in\nat$, $a_0, \ldots, a_{n-1} \in E$ and $a_{n} \in \mg E$.
Let $f(t) =\sum_{i=0}^n a_it^i \in E[t]$ be a polynomial and let $X\in F\setminus E$.
A valuation $w \in \Omega_v(F)$ is called \emph{overt with respect to the pair $(f, X)$} if $$w(f(X)) =\min \{w(a_jX^j)| 0\leq j \leq n\}, $$ and \emph{covert with respect to the pair $(f, X)$} otherwise.

Let $\Omega \subseteq \Omega_v(F)$. Then
\begin{itemize}
   \item $\mathsf{Overt}_{f,X}(\Omega) =\{ w\in \Omega\mid w \mbox{ is overt with respect to the pair } (f,X)\}$,  
 \item $\mathsf{Covert}_{f,X}(\Omega) =\{ w\in \Omega\mid w \mbox{ is covert with respect to the pair } (f,X)\}.$
   \end{itemize}
Note that $\Omega = \mathsf{Overt}_{f,X}(\Omega) \cup  \mathsf{Covert}_{f,X}(\Omega).$

Let  $\Omega_v^\ast(F)$ be the subset of $\Omega_v(F)$ consisting of those extensions $w$ for which $Fw/Ev$ is not ruled. Then
$$ \Omega_v^\ast (F) = \mathsf{Overt}_{f,X}(\Omega_v^\ast(F)) \cup  \mathsf{Covert}_{f,X}(\Omega_v^\ast(F)). $$

\Cref{overt} says that for $F =E(X)[\sqrt{f(X)}]$,
$|\mathsf{Overt}_{f,X}(\Omega_v^\ast(F))|\leq n^2$.

\section{Covert extensions}\label{4}
Let $E$ be a field with a valuation $v$. Let $F$ be a regular function field over $E$.
In this section, for a polynomial $g(t)\in E[t]$ and  $X\in F\setminus E$, we study covert extensions of $v$ to $F$ with respect to the pair $(g,X)$. 

Now we fix some notations, which are important for the rest of the article.

\begin{nota}
 Let $n\in\nat$, $a_0, \ldots, a_{n-1} \in E$ and $a_{n} \in \mg E$. 
Let $g(t) = \sum_{i=0}^n a_it^i \in E[t]$ be a polynomial and let $X\in F\setminus E$.

Let $S \subseteq \{0, 1, \ldots, n\}$. We set
$$\Omega_{S,g,X} =\{w \in \Omega_v (F) \mid S =S_{g,X}(w)\}.$$
Note that 
$$\Omega_v(F) = \bigcup_{S\subseteq\{0, 1,  \ldots,n\} }\Omega_{S,g,X}\,.$$
In view of \Cref{minimum-attained-twice}, if $|S| =1$, then $\Omega_{S,g,X}\cap \Omega^\ast_v(F) =\emptyset$. Since we want to bound $|\Omega^\ast_v(F)|$, from now on we focus on $S$ with $|S|\geq 2$.
\end{nota}


\begin{lem}\label{cS}
Let $n\in \nat$, $0 \leq i_0 < i_1 < \cdots < i_k\leq n$ and   $S = \{ i_0, i_1, \ldots, i_k\}$. Assume that $k\geq 1$.
Let $d$ be the greatest common divisor of $(i_1-i_0), \ldots, (i_k-i_0)$.
 Let $g(t)\in E[t]$ and $X \in F\setminus E$.  
There exists $b\in \mg E$ such that $w(bX^{d}) = 0$ for all $w \in \Omega_{S,g,X}$. 
\end{lem}

\begin{proof}
 Let $g(t) = \sum_{i=0}^ma_it^i$.  
Let $w \in\Omega_{S,g,X}$. Then $$S_{g,X}(w) = S = \{ i_0, i_1, \ldots, i_k\}.$$ 
For $1\leq j\leq k$, set $l_j = \frac{i_j-i_0}{d}$. We have 
$w(a_{i_j}X^{i_j}) =w(a_{i_0}X^{i_0})$
and hence  $$v  (a_{i_j} a_{i_0}^{-1}) = w (X^{i_0-i_j}) = - l_{i_j}w (X^{d}).$$
Since the greatest common divisor of $l_1, \ldots, l_k$ is~$1$, there are 
$m_1, \ldots, m_k \in\zz$ such that $l_1m_1 + \cdots+ l_km_k =1$. Then $$w(X^{d}) = -\sum_{j=1}^k m_jv(a_{i_j} a_{i_0}^{-1} ) \in vE.$$ 
In particular, for $b = a_{i_0}^{-(m_1+\ldots+m_k)} \prod_{j=1}^k a_{i_j}^{m_j}$, we have $w(bX^{d}) =0$.
\end{proof}

Let $S, b,d$ be as in \Cref{cS}.
Let $e_S$ be the order of $\frac{v(b)}{d}$ in $vE^\mathbb Q/vE$,  where $vE^\mathbb Q$ is the division closure of the value group $vE$. 
Note that, since for all $w\in \Omega_{S,g,X}$, $w(X) = -\frac{v(b)}{d}$, hence $e_S$ is also the order of $w(X)$ in $wF/vE$.
Thus there exists $c \in \mg E$ such that $w(cX^{e_S}) =0$, for all $w \in \Omega_{S,g,X}$. 
We set $\mathcal C_g(S) = c\mg{\mc O}_v$.

Let $c \in \mathcal C_g(S)$.
Let $1\leq j\leq k$.
Since $e_S$ is the order of $\frac{v(b)}{d}$, $e_S$ divides $d$, hence $e_S$ divides $(i_j-i_0)$.
Set $n_j =\frac{i_j-i_0}{e_S} \in \nat$.
We have $v(a_{i_j} a_{i_0}^{-1}) = n_jv(c) $, that is, $v(a_{i_j} a_{i_0}^{-1}c^{-n_j}) =0 $.
\begin{de}\label{residue-poly}
    For $S$,  $g$, and $c$ as above, we associate a polynomial over $Ev$ as follows: 
$$p_{S,g,c}(t) = \sum_{j=0}^{k} \ovl{a_{i_j} a_{i_0}^{-1}c^{-n_j}} t^{n_j}. $$
\end{de}
Note that $\deg(p_{S,g,c}(t)) = n_k$.

\begin{lem}\label{residue-root}
Let $S \subseteq \{ 0, 1, \ldots, n \}$ with $|S|\geq 2$. 
Let $n\in\nat$ and $a_0, \ldots, a_{n} \in E$. 
Let $g(t) = \sum_{i=0}^n a_it^i \in E[t]$ and let $X\in F\setminus E$. 
Let $w\in \Omega_{S,g,X}$ and $c\in \mathcal C_g(S)$. 
Then $\deg(p_{S,g,c}(t))\leq \max S$.
Furthermore, 
 $w \in \mathsf{Covert}_{g,X}(\Omega_v(F)) $  if and only if $\ovl{cX^{e_S}}$ in $Fw$ is a root of $p_{S,g,c}(t)$.  
\end{lem}
\begin{proof}
Let $S =\{i_0, i_1, \ldots, i_k\}$ be such that $i_0<i_1< \cdots< i_k$.
By the definition of $p_{S,g,c}(t)$, we have that $$\deg(p_{S,g,c}(t)) = n_k = \frac{i_k-i_0}{e_S}  \leq i_k = \max S.$$
We have that
$$(a_{i_0}X^{i_0})^{-1}g(X)=  \sum_{i=0}^n (a_ia_{i_0}^{-1}) X^{i-i_0}  .$$
Note that  $w(g(X)) > \min \{w (a_iX^i)\mid 0\leq i\leq n\}$ if and only if 
\begin{equation}\label{1} 
\ovl{\sum_{i=0}^n (a_ia_{i_0}^{-1}) X^{i-i_0}} =0 \mbox{ in } Fw .
\end{equation}
By the definition of $S_{g,X}(w)$, we get that $w ((a_ia_{i_0}^{-1}) X^{i-i_0}) >0$ for $i \notin S_{g,X}(w)$, 
and $w((a_ia_{i_0}^{-1}) X^{i-i_0}) =0$ for $i \in S_{g,X}(w)$.  
Thus \eqref{1} holds if and only if 
$$\sum_{i\in S_{g,X}(w)} \ovl{(a_ia_{i_0}^{-1}) X^{i-i_0}} =0 \in Fw.$$
Let $c\in \mathcal C_g(S)$. Then $w(cX^{e_S} ) =0$. Now the statement follows, since 
$$p_{S,g,c}(\ovl{cX^{e_S}}) = \sum_{j=0}^{k} \ovl{a_{i_j} a_{i_0}^{-1}c^{-n_j}} (\ovl{cX^{e_S}})^{n_j} = \sum_{i\in S_{g,X}(w)} \ovl{(a_{i}a_{i_0}^{-1}) X^{i-i_0}},$$
where $i_j-i_0 = n_je_S$ for $n_j \in \nat$.
\end{proof}

Let $\mathsf{Irr fac} (p_{S,g,c})$ be the set of all monic irreducible factors of $p_{S,g,c}(t)$ in $Ev[t]$. 
In view of \Cref{residue-root}, for $w \in  \mathsf{Covert}_{g,X}(\Omega_{S,g,X})$, we have that  $\ovl{cX^{e_S}}$ in $Fw$ is algebraic  over $Ev$ and 
its minimal polynomial over $Ev$ is in $\mathsf{Irr fac} (p_{S,g,c})$.
Note that, for $c,c' \in \mathcal C_g(S)$, the set of multiplicities of all the roots of $p_{S,g,c}$ and the set of multiplicities of all the roots of  $p_{S,g,c'}$ are equal and is denoted by $\mathsf{Mult}(S,g,X)$.

For $c \in \mathcal C_g(S)$ as defined in the paragraph after \Cref{cS} and $q \in \mathsf{Irr fac} (p_{S,g,c})$, we set
$$\Omega_{S, g, X,q} = \{w \in \Omega_{S,g,X} \mid q(\ovl{cX^{e_S}}) =0 \mbox{ in } Fw\}.$$

For $g(t)\in E[t]$ a polynomial of degree $n$ and $X\in F\setminus E$, we denote by $\mathcal S (g, X)$ the collection of all subsets $S \subseteq \{0,1, \ldots, n\}$ such that  $|S| \geq 2$ and $\Omega_{S,g,X} \neq \emptyset$.
Then $\mathcal S (g, X)$ is finite. From \Cref{residue-root}, we have
\begin{equation}\label{double-union}
 \mathsf{Covert}_{g,X}(\Omega_{S,g,X} ) =  \bigcup_{ q \in  \mathsf{Irr fac}(p_{S,g,c}(t))} \Omega_{S,g,X,q}.   
\end{equation}

\section{Residue multiplicities reduction}\label{5}
Let $E$ be a field with a valuation $v$. Let $F$ be a regular function field over $E$.
Let $f(t)\in E[t]$ be a polynomial of degree $n$ and  $X\in F\setminus E$.
Recall that for $S \subseteq \{0,1,\ldots, n\}$, the set $\Omega_{S,f,X}$ is the set of all residually transcendental extensions 
of $v$ to $F$ with $S_{f,X}(w) = S$.
In this section, we consider a set $S \in  \mathcal S (f, X) $, that is, a set $S \subseteq \{0,1,\ldots, n\}$ such that
$|S|\geq 2$ and $\Omega_{S,f,X} \neq \emptyset$, and calculate the associated polynomial $p_{S,f,c}$. 
For $q \in \mathsf{Irrfac}(p_{S,f,c})$, we will construct some field extension $F'/F$ with a 
 subset $\Omega $ of $\Omega_v(F')$ that has the property 
that it consists of exactly one extension $w'$ corresponding to each valuation $w \in \Omega_{S,f,X,q}$ and for this $w'$, we have $F'w' =Fw$.

\begin{lem}\label{purelyinseppoly}
Let $b_0, \ldots, b_n \in \mc O_v$, $b_n\neq 0$, be such that  $p(t) = b_nt^n+ \ldots + b_1t + b_0$ is a primitive polynomial and $\ovl{p}(t) \in Ev[t]$ is purely inseparable of degree $r$ with $\beta \in Ev$ as its root. 
Assume that $\car{Ev}=0$ or $\car{Ev}>\deg(p(t))$. 
Then $v(b_r)=0$, $v(b_i) >0$ for $i>r$, and $\beta=-  \frac{ \ovl{b_{r-1}}}{r\ovl{b_r}}$. 
In case $\beta\neq 0\in Ev$, we also have $v(b_i)=0$ for all $i$, $0\leq i \leq r$. 
\end{lem}
\begin{proof}
 Since $\deg{(\ovl{p}(t))} =r$, we have $\ovl{p}(t) = \ovl{b_r} t^r+ \ldots + \ovl{b_0} \in Ev[t]$. 
 Thus  $v(b_i)>0$ for $i>r$.
Since $\ovl{p}(t) \in Ev[t]$ is purely inseparable with $\beta\in Ev$ as its root, we get $\ovl{p}(t) = \ovl{b_r}(t- \beta)^r$ with $\ovl{b_r}\neq 0 \in Ev$. 
Furthermore, since  $\car{Ev}=0$ or $\car{Ev}>n$, using binomial expansion and comparing the coefficient of $t^i$ for $0 \leq i\leq r-1$, we obtain that  
\begin{equation}\label{eq1}
  \ovl{b_i} = \ovl{b_r} (-1)^{r-i} \binom{r}{r-i}\beta^{r-i} \neq 0.  \end{equation}
Writing \ref{eq1} for $i =r-1$, we get that $\ovl{b_{r-1}} =\ovl{b_r} (-1) r \beta$, which implies that 
$$\beta=-\frac{\ovl{b_{r-1}}}{r\ovl{b_r}}.$$  
In case $\beta\neq 0\in Ev$, we further have $v(b_i)=0$ for all $i$, $0\leq i \leq r$ by \ref{eq1}. 
\end{proof}

\begin{lem}\label{purelyinseparable}
 Let $f(t) =\sum_{i=0}^na_it^n \in E[t]$ with $a_n\neq 0$ and $X \in F\setminus E$. 
Assume that $\car{Ev}=0$ or $\car{Ev}>\deg{f}$.  
Let $S =\{0,1,\ldots,r\} \in \mathcal{S}(f,X)$ and $c\in \mathcal C_f(S)$. 
Assume that the polynomial $p_{S,f,c}(t)$ is purely inseparable of degree $r$ and let $q$ be the irreducible factor of $p_{S,f,c}(t)$. Let $w \in \Omega_{S,f,X,q}$. Then 
\begin{enumerate}
\item Then $\ovl{\left( \frac{ra_rX}{a_{r-1}}\right)} = -1$ in $Fw$.

\item Let $b_i =\frac{a_i}{a_0c^i}$ and $\theta =cX$. The polynomial 
    $$C(t) = \sum_{i=0}^{n-r+1} \binom{r-1+i}{i}b_{r-1+i}\;t^i \in \mc O_v[t]$$ 
is primitive and $\ovl{C}(t)$ is linear and $\ovl{C}(\ovl{\theta}) =0$ in $Fw$.

\item There is an irreducible primitive factor $Q(t)$ of $C(t)$ such that $\ovl{Q}(t) = \ovl{C}(t)$. Let $F_Q = F(t)/(Q(t)), \alpha = t + (Q(t))$ and $E_Q =E[\alpha]$. There is a unique extension $w_Q$ of $w$ to $F_Q$ such that $w_Q(\alpha) =0$.
For this $w_Q$, $\ovl{\alpha} =\ovl{\theta}$ in $F_Qw_Q$. In particular, $w_Q(\theta -\alpha) > 0$.

\item Let $X_1 = \theta - \alpha$, $g(t) =a_0^{-1}f(c^{-1}(t+\alpha)$ and $S_1 = S_{g,X_1}(w_Q)$. 
Then $\max{S_1} \leq r$.
 Suppose $| S_1| \geq 2 $. Then $\max ~\mathsf{Mult(S_1,g, X_1)} <r$.
 

\item For the set $\Omega = \{w' \in \Omega_v(F_Q) \mid w'|_F \in \Omega_{S, f, X,q}, 
w'(\alpha) =0\}$
we have $| \Omega_{S, f, X,q} \cap \Omega_v^\ast(F)| = |\Omega \cap \Omega_v^\ast(F_Q)|$. 
Furthermore, $$\Omega \cap \Omega_v^\ast(F_Q) \subseteq  
\bigcup_{\substack{{S_1 \in \mathcal S(g,X_1)}\\S_1 
\subseteq \{0,1, \ldots,r\}}}\Omega_{S_1, g,X_1}.$$ 
\end{enumerate}
\end{lem}

\begin{proof} $(1)$ Since $S= \{0,1, \ldots, r\}$, $e_S =1$. Let $c\in \mathcal C_f(S)$. Then $w(cX)=0$ and since $w\in \Omega_{S,f,X,q}$, we have that $q(\ovl{cX}) =0$ in $Fw$.
By the definition of $p_{S,f,c}$ we get that  $p_{S,f,c}(t) = \sum_{i=0}^r (\ovl{\frac{a_i}{a_0c^i}})t^i \in Ev[t]$. 

For $0 \leq i\leq n$, set $b_i = \frac{a_i}{a_0c^i}$. 
Let $p(t) = a_0^{-1}f(c^{-1}t) =\sum_{i=0}^n b_it^n$.
We get $p_{S,f,c}(t) =\ovl{p}(t)$ in $Ev[t]$.
By \Cref{purelyinseppoly}, we get that 
$\ovl{cX} = -\ovl{\left(\frac{b_{r-1}}{rb_r} \right)}=    -\ovl{ \left(\frac{ca_{r-1}}{ra_r}\right)}$ in $Fw$. Thus $\ovl{\left(\frac{ra_rX}{a_{r-1}}\right)} =-1$ in $Fw$.

$(2)$ Set $\theta =cX$. Consider the polynomial 
$$C(t) = \sum_{i=0}^{n-r+1} \binom{r-1+i}{i}b_{r-1+i}t^i.$$ 
Since $\car{Ev} =0$ or $\car{Ev}>n$ and $b_j\in \mc O_v$ for $r-1 \leq j \leq n$, we have $C(t)\in \mc O_v[t].$ 
Since $\ovl{p}(t) = \sum_{i=0}^n\ovl{b_i}t^n$ is purely inseparable polynomial of degree $r$ it follows by \Cref{purelyinseppoly} that $v(b_j)>0$ for $j>r$, $v(b_r) =v(b_{r-1}) =0$ and $v(r) =0$. Hence the residue polynomial is given by
$$\ovl{C}(t) = r \ovl{b_r} t + \ovl{b_{r-1}} = r \ovl{b_r}\left(t +\ovl{\frac{b_{r-1}}{r\ovl{b_r}}}\right) = r\ovl{b_r}(t-\ovl{\theta}) \in Fw.$$ 

$(3)$ Since $\ovl{C}(t)$ is linear, by \Cref{simple_ext}(1) there is a primitive irreducible factor $Q(t)$ of $C(t)$ such that $\ovl{Q}(t)=\ovl{C}(t)$.
By \Cref{simple_ext}$(3)$ there is a unique extension $w_Q$ of $w$ to $F_Q$ such that  $w_Q(\alpha)=0$, $\ovl{\alpha}= -\frac{\ovl{b_{r-1}}}{r\ovl{b_r}}$ in $F_Qw_Q$ and $F_Qw_Q=Fw$. Using $(2)$, we note that $$\ovl{\alpha} = - \frac{\ovl{b_{r-1}}}{r\ovl{b_r}} = \ovl{\theta}$$ in $F_Qw_Q$. Thus $w_Q(\theta -\alpha) >0$.

$(4)$ We have $ w_Q(X_1) =  w_Q(\theta -\alpha) >0$.  
For $0 \leq j\leq n$, set 
$$c_j =\sum_{k=j}^n b_k \binom{k}{j} \alpha^{k-j} \in E[\alpha].  $$ 
Let $p(t) = a_0^{-1}f(c^{-1}t)$ as in proof of $(1)$. Then
$g(t) = p(t+\alpha)= \sum_{j=0}^n c_j t^j \in E[\alpha][t].$
Since $w_Q(\alpha)=0$ and  $b_i\in \mc O_v$, we get that $w_Q(c_i)\geq 0$. 

By \Cref{purelyinseppoly}, we have that, for $r <i \leq n$, $v(b_i)>0 $ and $v(b_r) =0$.
Using this with  the definition of $c_i$ we get that for $r <i \leq n$, 
$v(c_i) >0$  and  $v(c_r) = v(b_r) =0 $.
This shows that $g(t)$ is primitive.
Since  $w_Q(X_1)>0$, for $i>r$, we observe that 
$$w_Q (c_iX_1^i ) >  w_Q(X_1^i) > w_Q(X_1^r ) = w_Q(c_rX_1^r).$$
Then $\min \{ w_Q(c_iX_1^i)\mid 0\leq i \leq n\} \leq w_Q(c_r X_1^r)$, 
whereby for $S_1= S_{g,X_1}(w_Q)$, we have $\max S_1 \leq r$.
Since $Q(\alpha) =0$ and  $Q(t)$ is a factor of $C(t)$, we get that $c_{r-1} = C(\alpha) =0$. 

Let $S_1 = \{i_0, \ldots, i_k\}  $ with $0 \leq i_0 < i_1< \cdots <i_k \leq r$. Assume that $|S_1|\geq2$. Let $e_{S_1}$ be the order of $w_Q(X_1)$ in $w_QF_Q/w_QE_Q$ and let $c_{S_1} \in \mathcal C_g(S_1)$.
Then $w_Q(c_{S_1}X_1^{e_{S_1}})=0$ and
$$p_{g,X_1,c_{S_1}}(t) = \sum_{j=0}^k \ovl{(c_{i_j}c_{i_0}^{-1} c_{S_1}^{-n_j})} t^{n_j}\in E[\alpha]w_Q \mbox{ where }  n_j = \frac{i_j-i_0}{e_{S_1}}$$
Since $ i_k = \max S_1  \leq r$ we get that $\deg(p_{g,X_1,c_{S_1}}(t))  \leq r$ . 
Assume that $r = \deg (p_{g,X_1,c_{S_1}}(t))$. Then $i_k =r$.
Since $c_{r-1} =0$, $\car{Ev} =0$ or $\car{Ev} > r$, and 
$p_{g,X_1,c_{S_1}}(t)$ has a non-zero root in $Ev$, we get by \Cref{purelyinseppoly} 
that $p_{g,X_1,c_{S_1}}(t)$ is not purely inseparable.

$(5)$ First statement follows by $(3)$. 
Let $w'\in \Omega \cap \Omega^{\ast}_v(F_Q)$. 
We have by $(3)$ and $(4)$, that  for $S_1=S_{g,X_1}(w')$, $S_1\subseteq \{0,1,\ldots, r\}$. Since $F'w'/Ev$ is non-ruled, we have $|S_1|\geq 2$ by \Cref{minimum-attained-twice}. Hence $w' \in \bigcup_{\substack{{S_1 \in \mathcal S(g,X_1)}\\S_1 
\subseteq \{0,1, \ldots,r\}}}\Omega_{S_1, g,X_1}.$
\end{proof}

\begin{lem}\label{separable}
Let $f(t) \in E[t]$ and $X \in F\setminus E$. 
Assume that $\car{Ev}=0$ or $\car{Ev}>\deg{f}$. 
Let $S \in \mathcal S(f,X)$. Let $e =e_S$
 and $c\in \mathcal C_f(S)$. 
 Let $q(t) \in \mathsf{Irrfac}(p_{S,f,c})$ and $r$ be the highest power of $q$ dividing $p_{S,f,c}$.
Let $w \in \Omega_{S,f,X,q}$. 
\begin{enumerate}
\item  Let $Q(t) \in  {\mc O}_v[t]$ be monic such that $\ovl{Q}(t) =q(t)$.
Consider $F_Q = F[t]/(Q(t))$,  $\alpha = t+ (Q(t))$ in $F_Q$ and $E_Q =E[\alpha]$.
There is a unique extension $w_Q$ of $w$ to $F_Q$ such that $w_Q(\alpha) =0$ and $\ovl{\alpha} = \ovl{cX^{e}}$ in $F_Qw_Q$. We have $w_Q F_Q = wF$. The polynomial $s^e- \alpha^{-1}c$ is irreducible over $F_Q$. 
\item Let $F' = F[\sqrt[e]{\alpha^{-1}c}]$, $\lambda \in F'$ such that $\lambda^{e} =\alpha^{-1}c$ and $E' =E_Q[\lambda]$. 
There is a unique extension $w'$ of $w$ to $F'$ such that $w'(\lambda X) =0$, $\ovl{\lambda X} = 1 $ in $F'w'$, and for the valuation $w'$, $F'w'=Fw$.
    \item
    The polynomial $p(t) = a_{i_0}^{-1}\lambda^{i_0}f(\lambda^{-1}t)$ is primitive in $\mc O_{w'}[t]$ and $\ovl{\lambda X}\in F'w'$ is  a root of $\ovl{p}(t)$ of multiplicity $r$.
    \item
    Set  $X_1  = \lambda X - 1$ and  $g(t) = f (\lambda^{-1}(t+1)) $. Then $\max(S_{ g,X_1}(w')) \leq r$.   
    \item Set $S_1 = S_{ g,X_1}(w')$. Suppose $|S_1|\geq 2$. Then either $\max ~\mathsf {Mult}(S_1,g,X_1)<r $, or $p_{S_1,g, c_{S_1}}$ is a purely inseparable polynomial of degree $r$. 

 \item For the set $$\Omega = \{w' \in \Omega_v(F') \mid w'|_F \in \Omega_{S, f, X,q}, w'(\lambda X) =0, ~\ovl{\lambda X} =1 \mbox{ in } F'w'\},$$
we have $| \Omega_{S, f, X,q} \cap \Omega_v^\ast(F)| = |\Omega \cap \Omega_v^\ast(F')|$. 
Furthermore, $$\Omega \cap \Omega_v^\ast(F') \subseteq  
\bigcup_{\substack{{S_1 \in \mathcal S(g,X_1)}\\S_1 
\subseteq \{0,1,\ldots,r\}}}\Omega_{S_1, g,X_1}.$$

\end{enumerate}
\end{lem}

\begin{proof}
$(1)$ Considering $\beta=\ovl{cX^{e}} \in Fw$ in \Cref{baseextension}, 
we get that there is a unique extension $w_Q$ of $w$ to $F_Q$ with $w_Q(\alpha) =0$, $\ovl{\alpha} = \ovl{cX^{e}}$ in $F_Qw_Q$ and $F_Q w_Q = Fw$. Moreover $w_QF_Q =wF$.
Let $v_Q =w_Q|_{E_Q}$. Then by \Cref{baseextension}, $v_Q$ is the unique extension of $v$ to $E_Q$ and $v_QE_Q =vE$.
Since $e$ is the order of $w(X)$ in $wF/vE$, it is also the order of $w(X)$ in $w_QF_Q/v_QE_Q$. Now, the irreducibility of $s^e-\alpha^{-1}c$ over $F_Q$ follows from \Cref{baseextension3}. 

$(2)$ For the valuation $w_Q$ on $F_Q$, we have that $\ovl{cX^e} = \ovl{\alpha} \in E_Qw_Q$ and hence applying \Cref{baseextension3},  we get a unique extension $w'$ of $w_Q$ to $F'$ such that $w'(\lambda X) =0$, $\ovl{\lambda X} =1$ and 
$F'w' =F_Qw_Q =Fw$.

$(3), (4)$ Set $\theta = \lambda X$. We have $w'(X_1) =w'(\theta-1) >0$. 
Let $r \in \nat$ be the multiplicity of $\ovl{cX^{e}}$  in $p_{S,f,c}(t)$.
Note that $r \leq \deg( p_{S,f,c}(t))$.
Recall that $$p_{S,f,c}(t) = \sum_{j=0}^k \ovl{a_{i_j}a_{i_0}^{-1}c^{-n_j} } t^{n_j} \in Ev [t]$$ where $en_j = i_j -i_0$.
We have 
$$ p(t) = a_{i_0}^{-1}\lambda^{i_0}f(\lambda^{-1}t) = \sum_{i=0}^n a_ia_{i_0}^{-1}\lambda^{i_0-i}t^{i} = \sum_{i=0}^n b_it^i \in E'[t]$$
where $b_i = a_ia_{i_0}^{-1}\lambda^{i_0-i }$. Note that  $w'(b_i) =0$ for $i \in S$ and $w'(b_i) >0$ for $i \notin S$. Thus $p(t)$ is a primitive polynomial in $\mc O_{w'}[t]$ and its residue polynomial is
$$\ovl{p} (t) = \sum_{i=0}^{n} \ovl{b_i} t^{i}  = t^{i_0}\sum_{j=0}^k \ovl{b_{i_j}}t^{i_j-i_0} =t^{i_0}\sum_{j=0}^k \ovl{b_{i_j}}t^{en_j} =t^{i_0} p_{S,f,c_S} (\ovl{\alpha} t^e) \in E'w'[t]. $$
We conclude that $\ovl{\theta} =1$ is a root of $\ovl{p}(t)$ of multiplicity equal to $r$.  

Write $$\ovl{p}(t) = (t -1 )^r q_0(t). $$
where $q_0(t) \in E'w'[t]$ such that $q_0(1) \neq 0$. 
We have $a_{i_0}^{-1}\lambda^{i_0}g(t) = p(t+1)$. Write
$ a^{-1}_{i_0}\lambda^{i_0} g(t) = \sum_{i=0}^n c_it^i$, where $c_i \in \mc O_{w'}[t]$.
Then $$ \sum_{i=0}^n \ovl{c_i}t^i=  \ovl{a^{-1}_{i_0}\lambda^{i_0}g}(t)=  \ovl{p}(t+1) = t^r q_0(t+1) \in E'w'[t].$$
This shows that $\ovl{c_r} = q_0(1) \neq 0$, and hence $w'(c_r) =0$.
Since $ w'(X_1)>0$, for $i>r$ we have
$$w' (c_iX_1^i ) \geq w'(X_1^i) > w'(X_1^r ) = w'(c_rX_1^r).$$
This shows that $S_{g, X_1}(w') \subseteq\{0, 1, \ldots, r\}$.

(5) Let $|S_1|\geq 2$. 
By (4), $\mathsf {Mult}(S_1,g,X_1)\leq r$. Suppose $\max ~\mathsf {Mult}(S_1,g,X_1)=r$. 
Then $p_{S_1,g,c_{S_1}}$ must be a purely inseparable polynomial of degree $r$, and $S_1=\{0, 1, \ldots, r\}$. 

(6) First statement follows from $(2)$. 
Let $w'\in \Omega \cap \Omega^{\ast}_v(F')$. 
By $(4)$, we have that for $S_1=S_{g,X_1}(w')$, $S_1\subseteq \{0, 1, \ldots, r\}$. Since $F'w'/Ev$ is non-ruled, we have $|S_1|\geq 2$ by \Cref{minimum-attained-twice}. Hence $w'$ belongs to
$\displaystyle \bigcup_{\substack{{S_1 \in \mathcal S(g,X_1)}\\S_1 
\subseteq \{0, 1, \ldots,r\}}}\Omega_{S_1, g,X_1}.$
\end{proof}

\section{Non-ruled residue extensions}\label{6} 

Let $E$ be a field and $v$ be a valuation on $E$. 
Let $f(X) \in E[X]$ be a square-free polynomial. 
Consider the hyperelliptic curve $\mc H : Y^2 = f(X)$. Let $F$ be the function field of $\mc H$ over $E$, that is, $F =E(X)[\sqrt{f(X)}]$.  

In this section, we prove that there are only finitely many residually transcendental extensions of $v$ to the function field $F$ for which the corresponding residue field extensions are non-ruled. In the case of $\deg(f)\leq 3$, it is shown in \cite{Ohm}, \cite{BG21}, \cite{RRTEC}. Usually the term hyperelliptic curve is used when $\deg(f) \geq 5$. However, the statements and proofs in this section also hold for the case when $\deg (f)<5$. 

\assume Let $n =\deg (f(X))$. We will assume throughout that $\car(Ev)=0$ or $\car(Ev)>n$.

Write $f(X)= \displaystyle \sum^{n}_{i=0} a_i X^i$
where $a_0, a_1, \ldots, a_n \in E$, $a_n \in \mg E$. Since $f(X)$ is square-free, note that both $a_0$ and $a_1$ cannot be simultaneously zero.

Recall that for $k\geq 1$, $S = \{i_0, \ldots, i_k\} \subseteq \{0, 1, \dots, n\}$ with $i_0 < \cdots < i_k$ and for $n_j = \frac{i_j-i_0 }{e_S}$ where $1\leq j\leq k$, and $c \in \mathcal C_f(S)$, we have the associated polynomial 
$$p_{S, f, c} (t)= \sum_{j=0}^k \ovl{a_{i_j}a_{i_0}^{-1}c^{-n_j}}t^{n_j} \in Ev[t].$$

\begin{prop}\label{base-case}
    Let $\Omega \subseteq \Omega_v(F)$.  
   Assume that  for every $w \in \Omega$,  $S =S_{f,X}(w)$, and $c_S\in E$ with $w(c_SX^{e_S}) =0$, the polynomial $p_{S,f,c_S}(t)$ is separable.
     Then $\mathsf{Covert}_{f,X}(\Omega)\cap \Omega_v^{\ast}(F) = \emptyset$. Moreover, 
     $\Omega \cap \Omega_v^{\ast}(F)$ is finite. 
\end{prop}

\begin{proof}
Write $f(t) = a_0 + a_1 t + \ldots+a_nt^n$.
Let $w \in \mathsf{Covert}_{f,X}(\Omega)$. We claim that $Fw/Ev$ is ruled.
Let $S =S_{f,X}(w)$ and $c_S \in \mathcal C_f(S)$. Then by \Cref{residue-root}, we get that $\ovl{c_SX^{e_S}}$ in $Fw$ is a root of $p_{S,f,c_S}(t)$.  Let $q(t) \in \mathsf{irrfac}(p_{S,f,c_S})$ be the minimal polynomial of $\ovl{c_SX^{e_S}}$ over $Ev$. 
Let $Q(t) \in \mc O_v[t]$ be monic such that $\ovl{Q}(t)=q(t) \in Ev[t]$. Set $F_Q = F[t]/(Q(t))$,  $\alpha = t+ (Q(t))$ in $F_Q$ and $F' = F_Q[\sqrt[e_S]{\alpha^{-1}c_S}]$.
(as in \Cref{separable}). Let $\lambda \in F'$ be such that $\lambda^{e_s} = \alpha^{-1}c_S$. Set $E' = E[\lambda]$ and $\theta  = \lambda X$. 

Let $w'$ be an extension of $w$ to $F'$ as in \Cref{separable}$(2)$.
We have $w'( \theta) =0$.  
Consider
$$h(t) = f (\lambda^{-1}t) = \sum_{i=0}^n a_i\lambda^{-i}t^i \in E'[t].$$ We have that $h(\theta) =f(X)$ and $E'(X) = E' (f(X))[\theta] = E'(h(\theta)) [\theta] = E'(\theta)$.
Let $i_0 = \min S$.

Consider the polynomial
$p_\theta (t) = a_{i_0}^{-1}\lambda^{i_0}(h(t) - f(X))   \in E'(f(X))[t]$. Since $f(X) = h(\theta)$, we get that $p_{\theta}(\theta)=0.$ 
Note that $$\deg (p_\theta) =n= \deg_{\theta}(h(\theta))= [E'(\theta):  E'(h(\theta))].$$ Hence $p_\theta$ is irreducible. 

Let  $p(t) = a_{i_0}^{-1}\lambda^{i_0} h(t) = a_{i_0}^{-1}\lambda^{i_0} f(\lambda^{-1}t) $. Then by \Cref{separable}(3), we have that $p(t)$ is primitive. 
Since $w \in \mathsf{Covert}_{f,X}(\Omega)$, $w' (a_{i_0}^{-1}\lambda^{i_0}f(X)) > 0$, we get that $p_\theta(t) \in \mc O_{w'}[t]$ is primitive.
Observe that $\ovl{p_{\theta}}(t)=\ovl{p}(t)$. 
Since $p_{S,f,c}(t)$ is separable, $\ovl{\theta}$ is a simple root of $\ovl{p_{\theta}}$ by \Cref{separable}(3).
Now, it follows by \Cref{ruled_applicable} that $F'w' /E'w'$ is ruled. Since $F'w' =Fw$, we get that $Fw/Ev$ is also ruled. Thus $\Omega^\ast_v(F) \cap \mathsf{Covert}_{f,X}(\Omega)  = \emptyset$.

By \Cref{overt}, we get that 
$\Omega^\ast_v(F) \cap \mathsf{Overt}_{f,X}(\Omega) $ is finite. Since $\Omega = \mathsf{Overt}_{f,X}(\Omega) \cup  \mathsf{Covert}_{f,X}(\Omega)$, we get that $\Omega \cap \Omega_v^{\ast}(F)$ is finite.
\end{proof}

\begin{ex}(Elliptic Curves)
Let $E$ be a field with valuation $v$. Assume that $\car{Ev}\neq 2,3$. 
Let $a,b\in E$ be such that $ab\neq 0$. 
Let $F=E(X)[\sqrt{X^3+aX+b}]$.

Then for any $w\in \Omega^{\ast}_v(F)$, there are four possibilities for $S:=S_{f,X}(w)$, which we list below and analyze the set $\Omega_{S,f,X}$. Let $c\in E^{\times}$ be such that $w(cX^{e_S})=0$ for all $w \in \Omega_{S,f,X} $.

\textbf{Case 1} $S=\{0, 1\}$. Here $i_0=0,i_1=1$. Hence $d=1$ and $e_S=1$. 
Then $p_{S,f,c}(t)=1+\ovl{ab^{-1}c^{-1}}t \in Ev[t]$.

\textbf{Case 2} $S=\{0,3 \}$.  Here $i_0=0,i_1=3$. Hence $d=3$ and $e_S=1$ or $3$. 

If $e_S=1$, then $p_{S,f,c}(t)=1+\ovl{b^{-1}c^{-3}}t^3 \in Ev[t]$.

If $e_S=3$, then $p_{S,f,c}(t)=1+\ovl{b^{-1}c^{-1}}t \in Ev[t]$.

\textbf{Case 3} $S=\{1,3 \}$. Here $i_0=1,i_1=3$. Hence $d=2$ and $e_S=1$ or $2$. 

If $e_S=1$, then $p_{S,f,c}(t)=1+\ovl{a^{-1}c^{-2}}t^2$.

If $e_S=2$, then $p_{S,f,c}(t)=1+\ovl{a^{-1}c^{-1}}t$.

\textbf{Case 4} $S=\{0, 1, 3\}$. Here $i_0=1,i_1=1, i_2=3$. Hence $d=1$ and $e_S=1$. Then $p_{S,f,c}(t)=1+\ovl{ab^{-1}c^{-1}}t+\ovl{b^{-1}c^{-3}}t^3 \in Ev[t]$.

In Cases 1,2, 3 and Case 4 with $\ovl{a^3/b^2} \neq -27/4$, the polynomial $p_{S,f,c}(t)$ is separable over $Ev$, by \Cref{base-case}, 
$\Omega_{S,f,X} \cap \Omega^\ast_v(F)$ is finite.

In Case 4 with $\ovl{a^3/b^2} = -27/4$, we have that 
$$ \ovl{bc^3}p_{S,f,c}(t) = t ^3 + \ovl{ac^2} t+ \ovl{bc^3} = \left(t-\ovl{\hbox{$\frac{3bc}{a}$}}\right)\left(t+\ovl{\hbox{$\frac{3bc}{2a}$}}\right)^2  .$$
The polynomial $p_{S,f,c}(t)$ is not separable over $Ev$ and  $\max ~\mathsf {Mult}(S,f, X)=2$. 
Take $X_1 = \frac{2a}{3b}X +1$ and $$g(t) = f\left(-\frac{3b}{2a}(t+1)\right) =  \left(-\frac{3b}{2a}\right)^3 \left(t^3 + 3t^2 + \frac{\Delta_{a,b}}{9b^2}t +\frac{\Delta_{a,b}}{27b^2} \right).$$
Note that, $\ovl{a^3/b^2} \neq -27/4$ implies that $v(\frac{\Delta_{a,b}}{b^2}) >0 $. 
Let $w \in \Omega_{S,f,X}$. We  have $w(X_1)> 0$ and $v(3)=0$, hence $w(X_1^3)> w(3X_1^2) $, whereby $\max{S_{g,X_1}(w)} \leq 2 $. Thus $\mathcal S (g,X_1) \subseteq \{ 0,1,2\}$. Let $\Delta_{a,b}:=4a^3+27b^2$. Note that
$w(\frac{\Delta_{a,b}}{9b^2}X_1) > w(\frac{\Delta_{a,b}}{9b^2})$. Thus 
$\mathcal S (g,X_1)$ consists of only one set $S_1 :=\{ 0,2\}$. Since $v(2) =0$, for any $c_1\in \mathcal C_g(S_1)$ (e.g. one can take $c_1 = \frac{b^2}{\Delta_{a,b}}$), the polynomial $p_{S_1,g, c_1}$ is separable.
Thus $\mathsf{Covert}_{g,X_1}(\Omega_{S,f,X})\cap \Omega_v^{\ast}(F) = \emptyset$.
We get that
$$\Omega_{S,f,X}\cap \Omega_v^{\ast}(F) =(\mathsf{Overt}_{f,X}(\Omega_{S,f,X})\cap \Omega_v^{\ast}(F))\cup (\mathsf{Overt}_{g,X_1}(\Omega_{S,f,X})\cap \Omega_v^{\ast}(F)), $$
and this is finite, by \Cref{overt}.
Thus we see that $|\Omega^{\ast}_v(F)|<\infty .$

In \cite{RRTEC}, more detailed investigation was made for this case and it was proved that $|\Omega^{\ast}_v(F)|\leq 1$. 

\end{ex}

\begin{thm}
\label{main-theorem}
There are finitely many residually transcendental extensions $w$ of $v$ to $F$ such that $Fw/Ev$ is non-ruled.
\end{thm}
\begin{proof} 
We have $F =E(X)[\sqrt{f(X)}]$.
Let $\mathcal S (f, X)$ be the collection of all subsets $S \subseteq \{0, 1, \ldots, n\}$ such that  $|S| \geq 2$ and $\Omega_{S,f,X} \neq \emptyset$.
By \Cref{minimum-attained-twice}, we have that 
$$\displaystyle \Omega_v^\ast(F) = \bigcup _{S \in\mathcal S (f, X) }(\Omega_{S,f,X} \cap \Omega_v^\ast(F)).$$
For each $S \in \mathcal S (f, X) $, we show that $|\Omega_{S,f,X} \cap \Omega_v^\ast(F)| < \infty$.
 We show this by induction on $\max~\mathsf{Mult}(S, f,X)$.
Fix $S \in \mathcal S(f,X)$ and an element $c_S \in \mathcal C_f(S)$. 
We consider the associated polynomial $p_{S,f,c_S}(t)$ in $Ev[t]$. Recall that $\mathsf{Mult}(S, f,X)$ is the set of multiplicities of all roots of the polynomial $p_{S,f,c_S}$.
 
 Suppose $\max~\mathsf{Mult}(S, f,X) =1$. 
Then the polynomial $p_{S,f,c_S}$ is separable.
Since  for each $w \in \Omega_{S,f,X}$ we have $S_{f,X}(w) =S$  and $p_{S,f,c_S}$ is separable, $\Omega_{S, f, X} \cap \Omega^\ast_v(F)$ is finite,  by \Cref{base-case}. 

\textbf{Induction hypothesis:} Let $E'/E$ be a finite extension, and let $ F'$ be a regular hyperelliptic function field over $E'$. 
Let $ X' \in  F'\setminus E'$ and $g(t) \in E'[t]$ be such that $ F' = E'(X')[\sqrt{g(X')}]$. 
Let $\Omega \subseteq \Omega_v(F')$. We assume that, for $1< r\leq  \deg(g)$ and for every valuation $w \in \Omega$ and $S = S_{g, X'}(w)$, 
if $\max~\mathsf {Mult}(S,g, X') <r$, then $ \Omega \cap \Omega_v^\ast(F')$ is finite.

Assume now that $\max~\mathsf{Mult}(S,f,X) =r$.
Let $ \tilde r = \max S$. Note that $r \leq \deg(p_{S,f, c_S}) \leq \max S -\min S \leq \tilde r$.


\textbf{Case$(a)$}: [$r= \tilde r$]. In this case, $S = \{0, 1, \ldots, r\}$ and  $p_{S,f,c_S}$ is purely inseparable of degree $r$. Let $q$ be the irreducible factor of $p_{S,f,c_S}$.  
In \Cref{purelyinseparable}, we constructed 
a finite field extension
$E_Q/E$, a polynomial $g \in E_Q[t]$ and $X_1 \in E_Q[X] $ such that $F_Q =FE_Q = E_Q(X)[\sqrt{g(X_1)}]$ and a set of valuations $\Omega \subseteq \Omega_v(F_Q)$ such that $| \Omega_{S, f, X, q} \cap \Omega_v^\ast(F)| = |\Omega \cap \Omega_v^\ast(F_Q)|$. Furthermore,
 $$\displaystyle \Omega \cap \Omega_v^\ast(F_Q) \subseteq \bigcup_{\substack{{S_1 \in \mathcal S(g,X_1)}\\S_1 \subseteq \{0, 1, \ldots, r\}}}\Omega_{S_1, g,X_1}.$$
Also, for each $S_1\subseteq \{0, 1, \ldots, r\} $ and $S_1 \in \mathcal S(g,X_1)$, we have $\max ~\mathsf {Mult}(S_1,g,X_1) < r $, and
 $\Omega_{S_1, g,X_1} \cap \Omega_v^\ast(F_Q)$ is finite, by the induction hypothesis. Thus
$|\Omega \cap \Omega^\ast_v(F_Q)| < \infty $. We conclude that
\begin{align*}
|\Omega_{S, f, X,q} \cap \Omega_v^\ast(F)| & = |\Omega \cap \Omega_v^\ast(F_Q)|<\infty.
\end{align*}
 In this case, we have that $\mathsf{Covert}_{f,X}(\Omega_{S,f,X}) =  \Omega_{S,f,X,q}$. Thus $\mathsf{Covert}_{f,X}(\Omega_{S,f,X}) \cap \Omega_v^\ast (F_Q)$ is finite.  
Also, $\mathsf{Overt}_{f,X}(\Omega_{S,f,X}) \cap \Omega_v^\ast (F_Q)$ is finite by \Cref{overt}. 
Thus $|\Omega_{S,f,X} \cap \Omega_v^\ast (F_Q)| < \infty$.

\textbf{Case$(b)$}: [$r < \tilde r$]. In this case,
$p_{S,f,c_S}$ is not purely inseparable or $\deg (p_{S,f,c_S} ) < \tilde r$.
We have that $\mathsf{Overt}_{f,X}(\Omega_{S,f,X}) \cap \Omega_v^\ast (F_Q)$ is finite by \Cref{overt}.
We now show that  $\mathsf{Covert}_{f,X}(\Omega_{S,f,X}) \cap \Omega_v^\ast (F_Q)$ is finite.

For $q \in  \mathsf{Irrfac}(p_{S,f,c_S})$, we consider the set  
$$\Omega_{S, f, X,q} = \{ w \in \Omega_{S,f,X} \mid  q(\ovl{c_S X^{e_S}}) =0 \mbox{ in } Fw\}.$$
Recall that by \Cref{residue-root}, we have 
$$\mathsf{Covert}_{f,X}(\Omega_{S,f,X}) = \bigcup_{ q \in  \mathsf{Irrfac}(p_{S,f,c_S})} \Omega_{S,f,X,q}. $$

Since there are finitely many monic irreducible factors of $p_{S,f,c_S}$, it is enough to show that, for every  $q \in  \mathsf{Irrfac}(p_{S,f,c_S})$, the set $ \Omega_{S,f,X,q} \cap \Omega_v^\ast (F)$ is finite. 

Let $q \in \mathsf{Irrfac}(p_{S,f,c_S})$. Let $r_q\in \nat$ be the highest power of $q$ dividing $p_{S,f,c_S}$. Then $r_q\leq r$. 
In \Cref{separable}, we constructed finite field extensions $E'/E$, $F' =FE'$ and a set of valuations $\Omega \subseteq \Omega_v(F')$ such that
$| \Omega_{S,f,X,q} \cap \Omega_v^\ast (F)|=|\Omega\cap \Omega_v^\ast(F')|$.
Furthermore, there is a polynomial $g \in E'[t]$ and $X_1 \in F'$ such that $F' =E'(X_1)[\sqrt{g(X_1)}]$ and   $\displaystyle \Omega \subseteq \bigcup_{\substack{{S_1 \in \mathcal S(g,X_1)}\\S_1 \subseteq \{0, 1, \ldots, r\}}}\Omega_{S_1, g,X_1}$.


Let $S_1 \subseteq \{0, 1, \ldots, r\}$ such that $S_1 \in \mathcal S(g,X_1)$ and let $c_{S_1}\in \mathcal C_g(S_1)$. 
If $p_{S_1,g,c_{S_1}}$ is not purely inseparable of degree $r$, then $\max ~\mathsf {Mult}(S_1,g,X_1) < r$, and hence $\Omega_{S_1,g,X_1} \cap \Omega_v^\ast(F')$ is finite, by the induction hypothesis. 

If $p_{S_1,g, c_{S_1}}$ is purely inseparable of degree $r$, then $S_1 =\{0,1,\ldots,r\}$, and it follows from Case $(a)$ that $\Omega_{S_1,g, X_1}\cap \Omega_v^\ast(F')$ is finite.
\end{proof}

\medskip

\textbf{Acknowledgements}: The second-named author would like to thank the Department of Science and Technology of India  for the INSPIRE Faculty fellowship grant IFA23- MA 197 and IIT Indore for Young Faculty Seed Grant IITIIYFRSG 2024-25/Phase-V/04R. The second-named author also thanks Shiv Nadar University for  supporting a research visit during this work. 
 The authors are grateful to Karim J. Becher and C\'edric A\"{i}d for their insightful discussions and valuable feedback on this work.

\bibliographystyle{amsalpha}

\end{document}